\documentclass[preprint,12pt]{article}

\usepackage[a4paper, total={5.8in, 7.8in}]{geometry}

\usepackage{amssymb}
\usepackage{amsthm}
\usepackage{amsmath}

\usepackage{tikz}
\usepackage{tkz-graph}
\usepackage{graphicx}
\usepackage[americanvoltages,fulldiodes,siunitx]{circuitikz}
\usepackage{caption}
\usepackage{subcaption}

\theoremstyle{plain}
\newtheorem{theorem}{Theorem}[section]
\newtheorem{corollary}[theorem]{Corollary}
\newtheorem{lemma}[theorem]{Lemma}
\newtheorem{proposition}[theorem]{Proposition}

\usepackage{amsthm}
\let\oldproofname=\proofname
\renewcommand{\proofname}{\rm\bf{\textit{\oldproofname}}}

\theoremstyle{definition}

\theoremstyle{remark}

\newtheorem*{note*}{Note}

\newcommand{\open}{``} 

\allowdisplaybreaks

\begin{document}

\title{The Braess' Paradox for Pendant Twins}

\author{Lorenzo Ciardo\thanks{Department of Mathematics, University of Oslo, Norway (lorenzci@math.uio.no)}}

\maketitle

\begin{abstract}
The Kemeny's constant $\kappa(G)$ of a connected undirected graph $G$ can be interpreted as the expected transit time between two randomly chosen vertices for the Markov chain associated with $G$. In certain cases, inserting a new edge into $G$ has the counter-intuitive effect of increasing the value of $\kappa(G)$. In the current work we identify a large class of graphs exhibiting this \open paradoxical" behavior -- namely, those graphs having a pair of twin pendant vertices. We also investigate the occurrence of this phenomenon in random graphs, showing that almost all connected planar graphs are paradoxical. To establish these results, we make use of a connection between the Kemeny's constant and the resistance distance of graphs.

\end{abstract}


\medskip


\medskip\medskip

\section{Introduction}
\label{sec_introduction}
A useful technique to get insight into the combinatorial structure of a graph is to imagine a random walker moving along its edges. The long-term and short-term behaviors of the discrete stochastic process associated with the walker are strictly linked to both local and global properties of the graph.

The so-called \textit{Kemeny's constant} $\kappa(G)$ of a connected graph $G$ (\cite{KemenySnell}) comes from this technique. Suppose that $G$ has $n$ vertices labeled $1,2,\dots,n$, and let $A$ be its adjacency matrix and $D=diag(Ae)$ be its diagonal degree matrix (where $e$ is the all ones vector in $\mathbb{R}^n$). We can consider a natural random walk defined as follows: the set of states is the vertex set $V(G)$; in any time step the walker moves from a vertex $v$ (the current position) to one of the neighbors of $v$, with all the neighbors being equally likely. The transition matrix of this random walk is then $T=D^{-1}A$. Letting $\sigma(T)=\{1,\lambda_2,\lambda_3,\dots,\lambda_n\}$ be the spectrum of $T$, the Kemeny's constant of $G$ is defined by
\begin{equation}
\label{kemeny_def_eigenvalues_1756_5sept}
\kappa(G)=\sum_{i=2}^n\frac{1}{1-\lambda_i}.
\end{equation}
This parameter provides information about how difficult it is for the random walker to travel in the graph, and how connected the graph is. More precisely, let the \textit{stationary distribution vector} $w$ be the unique entrywise positive vector in $\mathbb{R}^n$ satisfying $w^Te=1$ and $w^T T = w^T$ (see the paragraph about notation below). Then $\kappa(G)$ is the expected number of steps that the walker starting at vertex $i$ needs to reach vertex $j$ for the first time, when both $i$ and $j$ are randomly chosen according to the distribution $w$ (\cite{KemenySnell}).

The Kemeny's constant has met the interest of research communities in a variety of different settings, ranging from network architectures (\cite{Ghayoori-Leon-Garcia}) to consensus protocols (\cite{Jadbabaie-Olshevsky-first,Jadbabaie-Olshevsky-second}) and macroeconomics (\cite{Moosavi-Isacchini}).

A combinatorial expression for $\kappa(G)$ was found in \cite{KirklandZeng}. Let $m$ be the number of edges in $G$, $\tau$ be the number of its spanning trees, $d\in \mathbb{R}^n$ be its degree vector and $S=[s_{ij}]$ be its so-called \textit{$2$-forest matrix}, i.e., the symmetric $n\times n$ matrix whose $ij$-th entry $s_{ij}$ is the number of spanning forests consisting of two trees $T_1$ and $T_2$, such that $i\in V(T_1)$ and $j\in V(T_2)$ (we shall refer to such a forest as to an $\{i,j\}$-$2$-forest). Then
\begin{equation}
\label{kemeny_kirkland_formula_2004_3sept2019}
\kappa(G)=\frac{d^TS d}{4m\tau}.
\end{equation}

This combinatorial formula is a powerful tool to study a phenomenon known as the \textit{Braess' paradox for graphs}. Since the Kemeny's constant measures the expected length of a trip between two randomly chosen vertices of $G$, one could expect that inserting a new edge $e$ into $G$ would generate shortcuts, increase the connectivity and hence decrease $\kappa(G)$. Indeed, this is the observed behavior in the majority of cases. However -- quoting \cite{HuKirkland} -- \open new connections sometimes yield surprising results": for some special graphs we find that $\kappa(G \cup e)>\kappa(G)$. One example of these \open paradoxical" graphs is given by trees with at least $4$ vertices having a pair of \textit{twin pendant vertices} $a$ and $b$ (i.e., $a$ and $b$ are pendant vertices and they are both adjacent to a common vertex). In this case, adding the edge $ab$ results in an increase of the Kemeny's constant (\cite{KirklandZeng}; see also Theorem \ref{twin_braess_property_trees_31july_1651} below). 

The goal of the current work is to extend this result to a larger class of graphs, thus identifying new instances of the Braess' paradox. To this end, we will exploit an intriguing connection of the Kemeny's constant with the so-called \textit{resistance distance} (also known as \textit{effective resistance}) of graphs, which simulates the behavior of electrical resistance in electric circuits (\cite{Kirchhoff,Klein,KleinRandic,Koolen-Markowsky-Park}).

The rest of the paper is organized as follows. In Section \ref{sec_a_characterization_for_v_twin_braess_graphs} we analyze how the operation of adding an edge between two pendant twin vertices affects the value of the Kemeny's constant, and we give a criterion for a graph to be $v$-twin-Braess. In Section \ref{sec_resistance_distance} we use this criterion and certain properties related to the resistance distance of graphs to complete the proof of the first main result of the paper (Theorem \ref{theorem_main_connected_graphs_twin_braess_27_agosto_1340}). In Section \ref{sec_how_many_graphs_are_paradoxical} we focus on the probability that the Braess' paradox occurs in a random graph. In particular, we show that a large random connected planar graph is almost always paradoxical.\\ 

\underline{Notation}. We use the word \textit{graph} to denote a simple, undirected, unweighted graph. The vertex set of a graph $G$ is denoted by $V(G)$. We say that $G$ is \textit{nontrivial} if $|V(G)|>1$. A \textit{non-edge} of $G$ is an edge of the complement of $G$. $\mathbb{R}^n$ is the space of $n$-dimensional real column vectors, and we identify such vectors with the corresponding $n$-tuples. The $i$'th component of a vector $x$ is denoted by $x_i$. The transpose of $x$ is denoted by $x^T$. $M_n(\mathbb{R})$ denotes the space of real square matrices of order $n$.
\section{A characterization for $v$-twin-Braess graphs}
\label{sec_a_characterization_for_v_twin_braess_graphs}
Let $G$ be a connected graph on $n$ vertices and let $v\in V(G)$. We denote by $\tilde{G}_v$ the graph obtained from $G$ by attaching two pendant vertices at $v$. Also, we denote by $\hat{G}_v$ the graph obtained from $\tilde{G}_v$ by connecting the two additional vertices with an edge. We say that $G$ is \textit{$v$-twin-Braess} if $\kappa(\hat{G}_v)>\kappa(\tilde{G}_v)$ (Figure \ref{fig_example_v_twin_braess_1810_5sept}). If $G$ is $v$-twin-Braess for every $v\in V(G)$, then we say that $G$ is \textit{twin-Braess}. 
\begin{figure}
\centering
\begin{tikzpicture}[scale=.7]
%
%
\draw[fill]  (0,0) circle (0.1cm);
\draw[fill]  (1,1.73205) circle (0.1cm);
\draw[fill]  (1,-1.73205) circle (0.1cm);
\draw[fill]  (2,0) circle (0.1cm);
\draw  (2+.1,0+.4) node{$v$};
\draw  (1,2.5) node{$G$};
%
\draw[black] (0,0) -- (1,1.73205) -- (2,0) -- (1,-1.73205) -- (0,0);
\draw[black] (0,0) -- (2,0);
\begin{scope}[shift={(5,0)}]
%
%
\draw[fill]  (0,0) circle (0.1cm);
\draw[fill]  (1,1.73205) circle (0.1cm);
\draw[fill]  (1,-1.73205) circle (0.1cm);
\draw[fill]  (2,0) circle (0.1cm);
\draw[fill]  (2+1.73205,1) circle (0.1cm);
\draw[fill]  (2+1.73205,-1) circle (0.1cm);
\draw  (2+.1,0+.4) node{$v$};
\draw  (2,2.5) node{$\tilde G_v$};
%
\draw[black] (0,0) -- (1,1.73205) -- (2,0) -- (1,-1.73205) -- (0,0);
\draw[black] (0,0) -- (2,0);
\draw[black] (2,0) -- (2+1.73205,1);
\draw[black] (2,0) -- (2+1.73205,-1);
\end{scope}
\begin{scope}[shift={(11.73205,0)}]
%
%
\draw[fill]  (0,0) circle (0.1cm);
\draw[fill]  (1,1.73205) circle (0.1cm);
\draw[fill]  (1,-1.73205) circle (0.1cm);
\draw[fill]  (2,0) circle (0.1cm);
\draw[fill]  (2+1.73205,1) circle (0.1cm);
\draw[fill]  (2+1.73205,-1) circle (0.1cm);
\draw  (2+.1,0+.4) node{$v$};
\draw  (2,2.5) node{$\hat G_v$};
%
\draw[black] (0,0) -- (1,1.73205) -- (2,0) -- (1,-1.73205) -- (0,0);
\draw[black] (0,0) -- (2,0);
\draw[black] (2,0) -- (2+1.73205,1);
\draw[black] (2,0) -- (2+1.73205,-1);
\draw[black] (2+1.73205,+1) -- (2+1.73205,-1);
\end{scope}
\end{tikzpicture}
\caption{An example of the graphs $G$, $\tilde G_v$ and $\hat G_v$. Using either \eqref{kemeny_def_eigenvalues_1756_5sept} or \eqref{kemeny_kirkland_formula_2004_3sept2019} one finds that $\kappa(\tilde G_v)=4.6786$ and $\kappa(\hat G_v)=5.1354$, meaning that $G$ is $v$-twin-Braess.}
\label{fig_example_v_twin_braess_1810_5sept}
\end{figure}
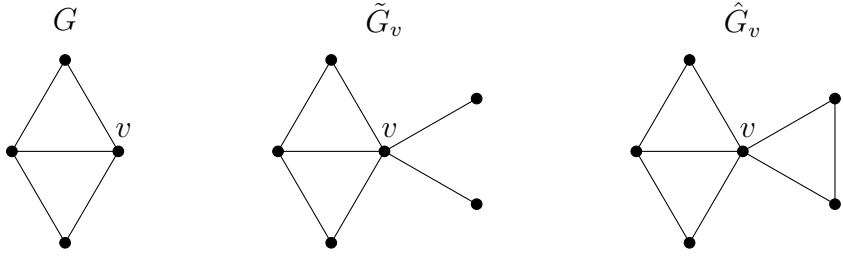

From \cite{KirklandZeng} we have the following result:
\begin{theorem}[\cite{KirklandZeng}]
\label{twin_braess_property_trees_31july_1651}
Every nontrivial tree is twin-Braess.
\end{theorem}
The first main result of this paper is an extension of Theorem \ref{twin_braess_property_trees_31july_1651} to general connected graphs.
\begin{theorem}
\label{theorem_main_connected_graphs_twin_braess_27_agosto_1340}
Every nontrivial connected graph is twin-Braess.
\end{theorem}
The following two sections are dedicated to the proof of Theorem \ref{theorem_main_connected_graphs_twin_braess_27_agosto_1340}. First, we provide a criterion to decide whether a given graph is $v$-twin-Braess for some fixed vertex $v$ (Proposition \ref{prop_criterion_v_twin_braess_2aug_1621}). Then, in Section \ref{sec_resistance_distance}, we use some properties of the resistance distance of graphs to conclude the proof.

We need the following lemma, which is straightforward to prove.
\begin{lemma}
\label{lemma_frazioni_1658_26_august}
Let $a_1,a_2,a_3,a_4$ be positive real numbers. Then
\begin{equation*}
\frac{a_1+a_2}{a_3+a_4}>\frac{a_1}{a_3}\hspace{.5cm}\Leftrightarrow\hspace{.5cm}
\frac{a_2}{a_4}>\frac{a_1}{a_3}.
\end{equation*}
\end{lemma}

Given a connected graph $G$ as above whose vertices are labeled $1,2,\dots,n$, we let $m$ be the number of its edges, $\tau$ be the number of its spanning trees, $d=(d_1,d_2,\dots,d_n)\in \mathbb{R}^n$ be the vector containing the degrees of its vertices and $S=[s_{ij}]\in M_n(\mathbb{R})$ be its $2$-forest matrix. Given a vertex $v\in V(G)$ we define 
\begin{equation}
\label{lambda_v_G_2aug_1223}
\lambda_v(G)=-3d^TS d+12md^TS e_v + 8m^2\tau +4m\tau -12\tau
\end{equation}
(where $e_v\in\mathbb{R}^n$ is the $v$-th unit vector). 

The next result shows that the sign of $\lambda_v(G)$ determines whether $G$ is $v$-twin-Braess or not.
\begin{proposition}
\label{prop_criterion_v_twin_braess_2aug_1621}
$G$ is $v$-twin-Braess if and only if $\lambda_v(G)>0$.
\end{proposition}
\begin{proof}
Consider the two graphs $\tilde{G}_v$ and $\hat{G}_v$ as defined above, and denote by $a$ and $b$ the two extra vertices attached to $v$ in $\tilde{G}_v$ and $\hat{G}_v$. Let $\tilde d\in\mathbb{R}^{n+2}$, $\tilde{S}=[\tilde{s}_{ij}]\in M_{n+2}(\mathbb{R})$, $\tilde m$ and $\tilde{\tau}$ be the degree vector of $\tilde{G}_v$, its $2$-forest matrix, the number of its edges and the number of its spanning trees respectively. Analogously, let $\hat d\in\mathbb{R}^{n+2}$, $\hat{S}=[\hat{s}_{ij}]\in M_{n+2}(\mathbb{R})$, $\hat m$ and $\hat{\tau}$ be the degree vector of $\hat{G}_v$, its $2$-forest matrix, the number of its edges and the number of its spanning trees respectively. We write $V$ instead of $V(G)$ for brevity.

Clearly, $\tilde{m}=m+2$ and $\hat{m}=m+3$. The degree vectors $\tilde d$ and $\hat d$ have the following description:
\begin{align*}
\tilde d_i=
\left\{
\begin{array}{cll}
d_i & \mbox{if} & i\in V\setminus \{v\}\\
d_v+2 & \mbox{if} & i=v\\
1 & \mbox{if} & i\in\{a,b\}
\end{array}
\right.
\hspace{.8cm}
\mbox{and}
\hspace{.8cm}
\hat d_i=
\left\{
\begin{array}{cll}
d_i & \mbox{if} & i\in V\setminus \{v\}\\
d_v+2 & \mbox{if} & i=v\\
2 & \mbox{if} & i\in\{a,b\}.
\end{array}
\right.
\end{align*}
The spanning trees of $\tilde{G}_v$ are precisely the spanning trees of $G$ with the addition of the edges $va$ and $vb$; hence, $\tilde{\tau}=\tau$. Moreover, every spanning tree of $\hat G_v$ is obtained from a spanning tree of $G$ by adding exactly two of the three edges $va$, $vb$ and $ab$, so that $\hat{\tau}=3\tau$. Similar arguments show that the elements of the $2$-forest matrices $\tilde S$ and $\hat S$ are given by
\begin{equation*}
\tilde s_{ij}=
\left\{
\begin{array}{cll}
s_{ij} & \mbox{if} & i,j\in V\\
0 & \mbox{if} & i=j\in\{a,b\}\\
2\tau &\mbox{if} & i,j\in\{a,b\}, i\neq j\\
\tau+ s_{vj} & \mbox{if} & i\in\{a,b\}, j\in V\\
\tau+ s_{iv} & \mbox{if} & i\in V, j\in\{a,b\}
\end{array}
\right. 
\end{equation*}
and
\begin{equation}
\label{s_hat_elements_1135_23_aug}
\hat s_{ij}=
\left\{
\begin{array}{cll}
3 s_{ij} & \mbox{if} & i,j\in V\\
0 & \mbox{if} & i=j\in\{a,b\}\\
2\tau &\mbox{if} & i,j\in\{a,b\}, i\neq j\\
3s_{vj}+2\tau & \mbox{if} & i\in\{a,b\}, j\in V\\
3s_{iv}+2\tau & \mbox{if} & i\in V, j\in\{a,b\}.
\end{array}
\right. 
\end{equation}
The fourth line of \eqref{s_hat_elements_1135_23_aug} comes from the fact that, if $i=a$ and $j\in V$, then an $\{i,j\}$-$2$-forest $(T_1,T_2)$ in $\hat G_v$ can be of any of the following three types: 
\begin{enumerate}
\item
$V(T_1)=\{a\}$, $T_2$ is a spanning tree of $G$ with the addiction of the edge $bv$;
\item
$V(T_1)=\{a,b\}$, $T_2$ is a spanning tree of $G$;
\item
$a,b,v\in V(T_1)$;  then $(T_1,T_2)$ is obtained by taking a $\{v,j\}$-$2$-forest $(U_1,U_2)$ in $G$ and adding to $U_1$ exactly two of the three edges $ab,av,bv$.
\end{enumerate}
The first two types amount to $\tau$ different $\{i,j\}$-$2$-forests each, while the third type amounts to $3s_{vj}$ different $\{i,j\}$-$2$-forests. The case when $i=b$ and the fifth line of \eqref{s_hat_elements_1135_23_aug} are analogous. 

Exploiting the symmetry of the $2$-forest matrix and the fact that its diagonal is zero, we have that
\begin{align*}
\tilde d^T\tilde S\tilde d&=\sum_{i,j\in V\setminus\{v\}}\tilde{d}_i\tilde d_j\tilde s_{ij}+2\sum_{j\in V\setminus\{v\}}\tilde{d}_v\tilde d_j\tilde s_{vj}+
2\sum_{\substack{i\in\{a,b\}\\j\in V\setminus\{v\}}}\tilde{d}_i\tilde d_j\tilde s_{ij}\\
&\quad+
2\sum_{j\in\{a,b\}}\tilde{d}_v\tilde d_j\tilde s_{vj}
+\sum_{i,j\in\{a,b\}}\tilde{d}_i\tilde d_j\tilde s_{ij}\\
&=\sum_{i,j\in V\setminus\{v\}}d_id_js_{ij}
+2\sum_{j\in V\setminus\{v\}}(d_v+2)d_js_{vj}
+
2\sum_{\substack{i\in\{a,b\}\\j\in V\setminus\{v\}}}d_j(\tau+s_{vj})\\
&\quad+2\sum_{j\in\{a,b\}}(d_v+2)(\tau+s_{vv})+4\tau\\
&=
\sum_{i,j\in V}d_id_js_{ij}+4\sum_{j\in V\setminus \{v\}}d_js_{vj}+4\tau\sum_ {j\in V\setminus\{v\}}d_j+
4\sum_ {j\in V\setminus\{v\}}d_js_{vj}\\
&\quad + 4\tau(d_v+2)+4\tau\\
&=
d^TSd+8 \sum_{j\in V\setminus \{v\}}d_js_{vj}+4\tau(2m-d_v)+4\tau d_v+12\tau\\
&=d^TSd+8d^TS e_v+8m\tau+12\tau
\intertext{and}
\hat d^T \hat S\hat d&=\sum_{i,j\in V\setminus\{v\}}\hat{d}_i\hat d_j\hat s_{ij}+2\sum_{j\in V\setminus\{v\}}\hat{d}_v\hat d_j\hat s_{vj}+
2\sum_{\substack{i\in\{a,b\}\\j\in V\setminus\{v\}}}\hat{d}_i\hat d_j\hat s_{ij}\\
&\quad+
2\sum_{j\in\{a,b\}}\hat{d}_v\hat d_j\hat s_{vj}
+\sum_{i,j\in\{a,b\}}\hat{d}_i\hat d_j\hat s_{ij}\\
&=3\sum_{i,j\in V\setminus\{v\}}d_id_js_{ij}
+6\sum_{j\in V\setminus\{v\}}(d_v+2)d_js_{vj}+
4\sum_{\substack{i\in\{a,b\}\\j\in V\setminus\{v\}}}d_j(3s_{vj}+2\tau)\\
&\quad+
4\sum_{j\in\{a,b\}}(d_v+2)(3s_{vv}+2\tau)
+16\tau\\
&=3\sum_{i,j\in V}d_id_js_{ij}
+12\sum_{j\in V\setminus\{v\}}d_js_{vj}
+24\sum_{j\in V\setminus\{v\}}d_js_{vj}\\
&\quad+16\tau\sum_{j\in V\setminus\{v\}}d_j+16\tau(d_v+2)
+16\tau\\
&=
3d^TSd+36\sum_{j\in V\setminus\{v\}}d_js_{vj}+16\tau(2m-d_v)+16\tau d_v+48\tau\\
&=
3d^TSd+36d^TS e_v+32m\tau+48\tau\\
&=3\tilde d^T\tilde S\tilde d + 12 d^TS e_v +8 m\tau + 12\tau.
\end{align*}
Moreover,
\begin{align*}
4\tilde m \tilde \tau &= 4(m+2)\tau =4m\tau+8\tau \hspace{2cm} 
\intertext{and}
4\hat m\hat \tau &=4(m+3)3\tau=12m\tau+36\tau=3(4\tilde m\tilde \tau)+12\tau.
\end{align*}
From expression \eqref{kemeny_kirkland_formula_2004_3sept2019} we obtain that
\begin{align*}
\kappa(\hat G_v)&=\frac{\hat d^T\hat S\hat d}{4\hat m\hat \tau}\\
&=\frac{3\tilde d^T\tilde S\tilde d + 12 d^TS e_v +8 m\tau + 12\tau}{3(4\tilde m\tilde \tau)+12\tau}.
\end{align*}
Using Lemma \ref{lemma_frazioni_1658_26_august} we see that
\begin{align*}
&\hspace{.9cm}\kappa(\hat G_v)>\kappa(\tilde G_v)\\
& \Leftrightarrow\hspace{.3cm}  \frac{12 d^TS e_v +8 m\tau + 12\tau}{12\tau}>\frac{\tilde d^T\tilde S\tilde d}{4\tilde m\tilde \tau}\\
& \Leftrightarrow\hspace{.3cm}  12(m+2)d^TSe_v+8m(m+2)\tau+12(m+2)\tau > 3\tilde d^T\tilde S\tilde d\\
& \Leftrightarrow\hspace{.3cm} 12md^TSe_v+24d^TSe_v+8m^2\tau+28m\tau +24\tau\\
&\hspace{.9cm} > 3 d^TSd+24 d^TSe_v+24 m\tau+36\tau\\
& \Leftrightarrow\hspace{.3cm} -3d^TSd+12md^TSe_v+8m^2\tau+4m\tau-12\tau >0\\
& \Leftrightarrow\hspace{.3cm} \lambda_v(G)>0,
\end{align*}
thus concluding the proof.
\end{proof}
\begin{note*}
We refer to \cite[Theorem 3.1.2.]{HuKirkland} for a different result on the variation in the Kemeny's constant of a graph when inserting an edge between two twin vertices (i.e., two nonadjacent vertices having the same neighborhood). There, the change in the Kemeny's constant is expressed in terms of the stationary vector and the mean first passage times of the quotient matrix related to an equitable partition of the transition matrix of the given graph.
\end{note*} 
\section{Resistance distance and a proof of Theorem \ref{theorem_main_connected_graphs_twin_braess_27_agosto_1340}}
\label{sec_resistance_distance}
Proposition \ref{prop_criterion_v_twin_braess_2aug_1621} shows that the problem of deciding whether a graph $G$ is $v$-twin-Braess can be reduced to studying the sign of the quantity $\lambda_v(G)$ defined in \eqref{lambda_v_G_2aug_1223}. In order to do so, we exploit the connection of the $2$-forest matrix $S$ with a particular metric on the vertex set of $G$: the so-called \textit{resistance distance} (\cite{Klein,KleinRandic,Koolen-Markowsky-Park}). We give its definition after a short digression on the Moore-Penrose inverse of a matrix. 

Given a matrix $M\in \mathbb{R}^{m,n}$, a \textit{Moore-Penrose inverse} of $M$ is a matrix $M^{\dagger}\in \mathbb{R}^{n,m}$ satisfying 
\begin{enumerate}
\item
$MM^\dagger M = M$
\item
$M^\dagger M M^\dagger = M^\dagger$
\item
$(MM^\dagger)^T=MM^\dagger$
\item
$(M^\dagger M)^T=M^\dagger M$.
\end{enumerate}
One can prove that such a matrix $M^\dagger$ always exists and is unique (\cite{Penrose}).
\begin{proposition}
\label{pos_semidef_moore_penrose_27_agosto_1216}
Let $M\in M_n(\mathbb{R})$ be a symmetric positive semidefinite matrix. Then $M^\dagger$ is symmetric positive semidefinite.
\end{proposition}
\begin{proof}
Transposing the equations characterizing the Moore-Penrose inverse we obtain
\begin{enumerate}
\item
$M(M^\dagger)^TM=M$
\item
$(M^\dagger)^TM(M^\dagger)^T=(M^\dagger)^T$
\item
$((M^\dagger)^TM)^T=(M^\dagger)^TM$
\item
$(M(M^\dagger)^T)^T=M(M^\dagger)^T$.
\end{enumerate}
This shows that $(M^\dagger)^T$ is a Moore-Penrose inverse of $M$, and by uniqueness we conclude that $(M^\dagger)^T=M^\dagger$. Moreover, given $x\in \mathbb{R}^n$,
\begin{align*}
x^TM^\dagger x&=x^TM^\dagger M M^\dagger x = (M^\dagger x)^T M (M^\dagger x)\geq 0,
\end{align*}
thus showing that $M^\dagger$ is positive semidefinite.
\end{proof}
Let now $G$ be a connected graph with $n$ vertices and consider the following matrices in $M_n(\mathbb{R})$: the Laplacian matrix $L$ of $G$; the Moore-Penrose inverse $L^\dagger$ of $L$; the diagonal matrix $\nabla$ such that $(\nabla)_{ii}=(L^\dagger)_{ii}$, $i=1,2,\dots,n$; the all ones matrix $J$. The \textit{resistance distance matrix} $\Omega=[\omega_{ij}]$ of $G$ is the $n\times n$ matrix defined by
\begin{equation}
\label{eqn_resistance_distance_as_pseudo_inverse}
\Omega=\nabla J + J\nabla -2 L^\dagger.
\end{equation}
The name comes from the fact that, if we consider the graph $G$ to be an electric circuit, every edge being a resistor with electrical resistance equal to $1$, then the resistance distance $\omega_{ij}$ between vertex $i$ and vertex $j$  in $G$ is the effective resistance in the circuit when we connect a battery across $i$ and $j$ (Figure \ref{fig_example_electric_graph_1942_5sept}). Formula \eqref{eqn_resistance_distance_as_pseudo_inverse} can then be derived from Kirchhoff's laws for electric circuits (as done in \cite[Corollary A]{KleinRandic}).
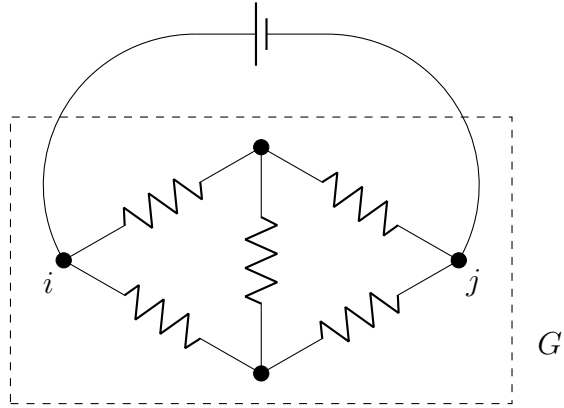
\begin{figure}
\centering
\begin{circuitikz}[baseline = {(0,-.5)}, scale = 1, transform shape]
%
%
\draw[fill]  (0,0) circle (0.1cm);
\draw[fill]  (2.5981,1.5) circle (0.1cm);
\draw[fill]  (2.5981,-1.5) circle (0.1cm);
\draw[fill]  (5.1962,0) circle (0.1cm);
\draw  (-.2,-.3) node{$i$};
\draw  (5.1962+.2,-.3) node{$j$};
\draw  (5.1962+.7+.5,-1.5-.4+.8) node{$G$};
\draw (0,0)
      to[R] (2.5981,1.5) 
      to[R] (5.1962,0) 
      to[R] (2.5981,-1.5)
      to[R] (0,0);
\draw (2.5981,1.5) to[R] (2.5981,-1.5);

\draw[] (0,0) arc (210:90:2);
\draw (1.73205,3) to[battery1] (5.1962-1.73205,3);
\draw[] (5.1962,0) arc (-30:90:2);
\draw[dashed]  +(-.7,-1.5-.4) rectangle +(5.1962+.7,1.5+.4);

%
%
%
\end{circuitikz}
\caption{A graph $G$ on $4$ vertices represented as an electric circuit, with a battery connected across vertices $i$ and $j$. Using either of the two expressions \eqref{eqn_resistance_distance_as_pseudo_inverse} and \eqref{eqn_S_tau_omega} one finds that $\omega_{ij}=1$.}
\label{fig_example_electric_graph_1942_5sept}
\end{figure}
As proved in \cite[Theorem 4. and (5)]{Bapat} (see also \cite{Barrett-Evans-Francis-Kempton-Sinkovic}), there is a close connection between the resistance distance matrix and the $2$-forest matrix of a graph. If $\tau$ is the number of spanning trees of $G$ and $S$ is its $2$-forest matrix, then
\begin{equation}
\label{eqn_S_tau_omega}
S=\tau\Omega.
\end{equation}
We now have all the ingredients to continue the study of $\lambda_v(G)$.

\begin{proposition}
\label{prop_lambda_regular_positive_2aug_1621}
Let $G$ be a nontrivial connected graph. Then $\lambda_v(G)>0$ $\forall v\in V(G)$.
\end{proposition}
\begin{proof}
If $m=1$ we have that $\tau=1$, $d=\begin{pmatrix}
1 \\ 
1
\end{pmatrix} $ and $S=\begin{pmatrix}
0 & 1 \\ 
1 & 0
\end{pmatrix} $. We obtain from \eqref{lambda_v_G_2aug_1223} that
\begin{align*}
\lambda_v(G)&=-3d^TS d+12md^TS e_v + 8m^2\tau +4m\tau -12\tau\\
&=-6+12+8+4-12=6>0.
\end{align*}
If $m\geq 2$ we deduce that $8m^2\tau +4m\tau -12\tau>0$, and hence
\begin{align*}
\lambda_v(G)&=-3d^TS d+12md^TS e_v + 8m^2\tau +4m\tau -12\tau\\
 &> -3d^TS d+12md^TS e_v.
\end{align*}
As a consequence, it is enough to prove that 
\begin{equation}
\label{claim_part_lambda_v_27_agosto_1329}
4md^TS e_v - d^TS d \geq 0.
\end{equation}
Letting $\Omega$, $L$, $\nabla$ and $J$ be as above and using \eqref{eqn_S_tau_omega} and \eqref{eqn_resistance_distance_as_pseudo_inverse} we obtain that
\begin{align*}
d^TS d&=\tau d^T\Omega d\\
&=\tau(d^T\nabla Jd+d^TJ\nabla d-2d^T L^\dagger d)\\
&=\tau(2md^T\nabla e+2me^T\nabla d-2d^TL^\dagger d)\\
&=\tau(4md^T\nabla e-2d^T L^\dagger d)
\end{align*}
and
\begin{align*}
4md^TS e_v&=4m\tau d^T\Omega e_v\\
&=4m\tau(d^T\nabla J e_v+d^TJ\nabla e_v-2d^T L^\dagger e_v)\\
&=4m\tau (d^T\nabla e+2me^T\nabla e_v-2d^T L^\dagger e_v)\\
&=4m\tau(d^T\nabla e+2m{L^\dagger}_{vv}-2d^T L^\dagger e_v).
\end{align*}
Hence,
\begin{equation*}
4md^TS e_v-d^TS d = 8m^2\tau{L^\dagger}_{vv}-8m\tau d^T L^\dagger e_v+2\tau d^T L^\dagger d. 
\end{equation*}
Since the Laplacian matrix $L$ is symmetric and positive semidefinite, Proposition \ref{pos_semidef_moore_penrose_27_agosto_1216} shows that so is $L^{\dagger}$. Introducing the two vectors $x=2m\sqrt{2\tau}e_v$ and $y=\sqrt{2\tau}d$, we deduce that
\begin{align*}
0&\leq (x-y)^TL^\dagger(x-y)\\&=x^TL^\dagger x-2y^TL^\dagger x + y^TL^\dagger y\\
&=8m^2\tau {L^\dagger}_{vv}-8m\tau d^T L^\dagger e_v + 2\tau d^TL^\dagger d\\
&=4md^TS e_v-d^TS d.
\end{align*}
This shows that \eqref{claim_part_lambda_v_27_agosto_1329} holds, thus concluding the proof of the proposition.
\end{proof}
\begin{proof}[\textbf{Proof of Theorem \ref{theorem_main_connected_graphs_twin_braess_27_agosto_1340}}]
The result follows directly from Proposition \ref{prop_criterion_v_twin_braess_2aug_1621} and Proposition \ref{prop_lambda_regular_positive_2aug_1621}.
\end{proof}
\section{How many graphs are paradoxical?}
\label{sec_how_many_graphs_are_paradoxical}
Given a connected graph $G$ and a non-edge $e$ of $G$, we say that $G$ is \textit{$e$-paradoxical} if $\kappa(G\cup e)>\kappa(G)$; we say that $G$ is \textit{paradoxical} if it is $e$-paradoxical for some non-edge $e$. Notice that a graph $G$ is $v$-twin-Braess for some fixed $v\in V(G)$ if and only if $\tilde G_v$ is $ab$-paradoxical, where $\tilde{G}_v$ is the graph obtained from $G$ by attaching two pendant vertices at $v$, and $a,b$ are these additional vertices. Recalling that two vertices in a graph are pendant twin vertices if they are both pendant and adjacent to a common vertex, we can then rewrite Theorem \ref{theorem_main_connected_graphs_twin_braess_27_agosto_1340} as follows:
\begin{corollary}
\label{cor_main_result_paradoxical}
Let $G$ be a connected graph with at least $4$ vertices, and let $a,b$ be twin pendant vertices. Then $G$ is $ab$-paradoxical.
\end{corollary}
It is natural, at this point, to wonder how many graphs have the property of being paradoxical. Recalling the interpretation of the Kemeny's constant given in the Introduction, we see that this question has practical importance. Suppose, for example, that a graph $G$ represents the plan of a ward in a hospital, each edge in the graph being a corridor. Consider the unlucky event of a patient developing an undiagnosed, highly contagious disease. His walk -- which we assume to be random in the sense described in the Introduction -- along the corridors of the hospital can then be very risky for other patients. In this setting, the Kemeny's constant of $G$ provides a measure of the estimated spreading rate of the infection: a higher value of $\kappa(G)$ is desirable, for it corresponds to a safer hospital. This should be taken into account when drawing the plans for the building. If $G$ is non-paradoxical, then the only way to obtain a safer ward is to delete a corridor, which would however limit the freedom of movement for patients. If $G$ is paradoxical, however, it is possible to increase both safety and freedom of movement by adding a corridor in some strategic position.

In case $G$ is a tree, a result from \cite{KirklandZeng} makes the situation interesting for the architects of the hospital.
\begin{theorem}[\cite{KirklandZeng}]
\label{thm_almost_all_trees_paradoxical}
Almost all unlabeled trees are paradoxical.
\end{theorem}
The expression \open almost all" here is borrowed from random graph theory: if $a_n$ is the number of unlabeled trees with $n$ vertices and $b_n$ is the number of unlabeled paradoxical trees with $n$ vertices, then the statement of Theorem \ref{thm_almost_all_trees_paradoxical} means that $\lim_{n\rightarrow\infty}\frac{b_n}{a_n}=1$. 

Theorem \ref{thm_almost_all_trees_paradoxical} is a consequence of Theorem \ref{twin_braess_property_trees_31july_1651} and of a result from \cite{Schwenk} implying that almost all unlabeled trees (in the sense described above) have a pair of twin pendant vertices. Letting $G$ be a generic connected graph, the situation becomes more complicated. The main reason for this, as pointed out in \cite{Schwenk}, is that the asymptotic behavior of graphs appears to be less structured than the asymptotic behavior of trees. One way to make things easier is to focus on planar graphs. In the example given above, this corresponds to a single floor hospital -- a reasonable limitation. Moreover, being the asymptotic structure of the automorphism group of random graphs vastly different from that of trees, we consider the planar graphs to be labeled. We can then give the second main result of the paper.
\begin{theorem}
\label{thm_almost_all_planar_graphs_paradoxical_5_september}
Almost all labeled connected planar graphs are paradoxical.
\end{theorem}
To prove Theorem \ref{thm_almost_all_planar_graphs_paradoxical_5_september} we use the work of McDiarmid, Steger and Welsh on asymptotic properties of random graphs (\cite{McDiarmid-Steger-Welsh-B,McDiarmid-Steger-Welsh-A}). 

We first set some terminology from \cite{McDiarmid-Steger-Welsh-B}. Let $\mathcal{A}$ be a class of labeled graphs closed under isomorphism, and, for $n\geq 1$, let $\mathcal{A}_n$ be the set of graphs in $\mathcal{A}$ on $n$ vertices (labeled $1,2,\dots,n$). We require that $\mathcal{A}_n\neq \emptyset$ for all sufficiently large $n$. We write $G_n\in_{u}\mathcal{A}_n$ to mean that $G_n$ is sampled from $\mathcal{A}_n$ uniformly at random (i.e., every graph in $\mathcal{A}_n$ has the same probability of being picked). We say that $\mathcal{A}$ has a \textit{growth constant} $\gamma$ if
\begin{equation*}
\lim_{n\rightarrow\infty}\left(\frac{|\mathcal{A}_n|}{n!}\right)^{\frac{1}{n}}=\gamma.
\end{equation*}
Consider two graphs $H$ and $G$ with $V(H)=\{1,2,\dots,h\}$ and $V(G)=\{1,2,\dots,n\}$, where $n>h$. Given a set $W\subset V(G)$ with $|W|=h$, we say that $H$ \textit{appears} at $W$ in $G$ if the increasing bijection from $\{1,2,\dots,h\}$ to $W$ gives an isomorphism from $H$ to the induced subgraph $G[W]$, there is a unique edge in $G$ joining $W$ to the rest of $G$ and this edge is incident with the least element in $W$. Denote by $a_H(G)$ the number of sets $W\subset V(G)$ such that $H$ appears at $W$ in $G$. Finally, if $H$ is a fixed graph with a distinguished vertex (root) $r$, we say that $H$ can be \textit{attached} to $\mathcal{A}$ if, for each $G\in\mathcal{A}$ and for each $v\in V(G)$, the graph obtained from $G \cup H$ by adding the edge $vr$ belongs to $\mathcal{A}$.

The following result shows that, under some conditions, a fixed graph $H$ appears  linearly many times in a random graph $G_n\in_u\mathcal{A}_n$ with high probability.
\begin{theorem}[\cite{McDiarmid-Steger-Welsh-B}]
\label{thm_pendant_appearances_1320_5sept}
Let $\mathcal{A}$ have a finite positive growth constant, ant let $G_n\in_u\mathcal{A}_n$. Let $H$ be a graph with a distinguished vertex $r$ that can be attached to $\mathcal{A}$. Then there exist constants $\alpha,\beta>0$ such that
\begin{equation*}
\mathbb{P}[a_H(G_n)\leq \alpha n]\leq e^{-\beta n}
\end{equation*}
for all sufficiently large $n$.
\end{theorem}
Gim\'{e}nez and Noy (\cite{Gimenez-Noy}) found the following asymptotic estimate for the number $c_n$ of labeled connected planar graphs on $n$ vertices: 
\begin{equation}
\label{eqn_number_connected_plaanr_graphs_1559_5sept}
c_n\sim c \cdot n^{-\frac{7}{2}}\gamma^n n!
\end{equation}
with $c\approx 4.1044 \cdot 10^{-6}$ and $\gamma\approx 27.2269$. This and Theorem \ref{thm_pendant_appearances_1320_5sept} are the last key elements we need to prove Theorem \ref{thm_almost_all_planar_graphs_paradoxical_5_september}.
\begin{proof}[\textbf{Proof of Theorem \ref{thm_almost_all_planar_graphs_paradoxical_5_september}}]
In the statement of Theorem \ref{thm_pendant_appearances_1320_5sept}, consider $\mathcal{A}$ being $\mathcal{CP}$, the class of labeled connected planar graphs, and $H$ being the path on $3$ vertices with the central vertex as root. From \eqref{eqn_number_connected_plaanr_graphs_1559_5sept} we have that $\mathcal{CP}$ has a finite positive growth constant. Moreover, it is clear that $H$ can be attached to $\mathcal{CP}$. Then Theorem \ref{thm_pendant_appearances_1320_5sept} implies that a labeled connected planar graph on $n$ vertices (sampled from $\mathcal{CP}_n$ uniformly at random) has a pair of twin pendant vertices -- actually linearly many -- with probability approaching $1$ as $n$ goes to infinity. Corollary \ref{cor_main_result_paradoxical} is then enough to conclude the proof of the theorem.
\end{proof}

\bigskip
\section*{Acknowledgments}
The author wishes to thank Geir Dahl (University of Oslo) and Wayne Barrett (Brigham Young University) for interesting discussions that helped to improve the final version of this work.


\begin{thebibliography}{99}



\bibitem{Bapat}  R.B. Bapat, Resistance distance in graphs, Math. Student 68(1-4) (1999): 87--98.

\bibitem{Barrett-Evans-Francis-Kempton-Sinkovic} W. Barrett, E.J. Evans, A.E. Francis, M. Kempton, J. Sinkovic, Spanning $2$-Forests and Resistance Distance in $2$-Connected Graphs, arXiv preprint arXiv:1901.00053 (2018).

\bibitem{Ghayoori-Leon-Garcia} A. Ghayoori, A. Leon-Garcia, Robust network design, 2013 IEEE International Conference on Communications (ICC), Budapest (2013): 2409--2414.

\bibitem{Gimenez-Noy} O. Gim\'{e}nez, M. Noy, Asymptotic enumeration and limit laws of planar graphs, Journal of the American Mathematical Society 22(2) (2009): 309--329.

\bibitem{HuKirkland} Y. Hu, S. Kirkland, Complete multipartite graphs and Braess edges, Linear Algebra and its Applications (2019).

\bibitem{Jadbabaie-Olshevsky-first} A. Jadbabaie, A. Olshevsky, On performance of consensus protocols subject to noise: Role of hitting times and network structure, 2016 IEEE 55th Conference on Decision and Control (CDC), Las Vegas (2016): 179--184.

\bibitem{Jadbabaie-Olshevsky-second} A. Jadbabaie, A. Olshevsky, Scaling laws for consensus protocols subject to noise, IEEE Transactions on Automatic Control 64(4) (2018): 1389--1402.

\bibitem{KemenySnell} J. Kemeny, J. Snell, \textit{Finite Markov Chains}, Springer-Verlag, New York (1976).

\bibitem{Kirchhoff} G. Kirchhoff, Ann. Phys. Chem. 72 (1847): 497--508.

\bibitem{KirklandZeng} S. Kirkland, Z. Zeng, Kemeny's Constant And An Analogue Of Braess' Paradox For Trees, Electronic Journal of Linear Algebra 31(1) (2016): 444--464.

\bibitem{Klein} D.J. Klein, Resistance-distance sum rules, Croatica chemica acta 75(2) (2002): 633--649.

\bibitem{KleinRandic} D.J. Klein, M. Randi\'{c}, Resistance distance, Journal of mathematical chemistry 12(1) (1993): 81--95.

\bibitem{Koolen-Markowsky-Park} J.H. Koolen, G. Markowsky, J. Park, On electric resistances for distance-regular graphs, European Journal of Combinatorics 34(4) (2013): 770--786.

\bibitem{McDiarmid-Steger-Welsh-B} C. McDiarmid, A. Steger, D.J.A. Welsh, Random graphs from planar and other addable classes, Topics in discrete mathematics, Springer, Berlin, Heidelberg (2006): 231--246.

\bibitem{McDiarmid-Steger-Welsh-A} C. McDiarmid, A. Steger, D.J.A. Welsh, Random planar graphs, Journal of Combinatorial Theory, Series B 93(2) (2005): 187--205.

\bibitem{Moosavi-Isacchini} V. Moosavi, G. Isacchini, A Markovian model of evolving world input-output network, PloS one 12(10) (2017): e0186746.

\bibitem{Penrose} R. Penrose, A generalized inverse for matrices, Mathematical proceedings of the Cambridge philosophical society 51(3) (1955): 406--413.

\bibitem{Schwenk} A.J. Schwenk, Almost all trees are cospectral, New directions in the theory of graphs (Proceedings of the Third Ann Arbor Conference on Graph Theory, University of Michigan, 1971), Academic Press, New York (1973): 275--307.



\end{thebibliography}
\end{document}